\documentclass[12pt]{amsart}
\usepackage{amssymb, amstext, amscd, amsmath}
\makeatletter
\def\@cite#1#2{{\m@th\upshape\bfseries%
[{#1\if@tempswa{\m@th\upshape\mdseries, #2}\fi}]}}
\makeatother
\theoremstyle{plain}
\newtheorem{thm}{Theorem}[section]

\newtheorem{prop}[thm]{Proposition}
\newtheorem{lem}[thm]{Lemma}
\theoremstyle{definition}
\newtheorem{rem}[thm]{Remark}
\newtheorem{defn}[thm]{Definition}

\newcommand{\eqbx}[1]{\medbreak\hfill\(\displaystyle #1\)\end{proof}}

\newcommand{\ca}{\mathrm{C}^*}

\newcommand{\rip}{\rangle}

\newcommand{\sot}{\textsc{sot}}

\newcommand{\bbC}{{\mathbb{C}}}

\newcommand{\bbJ}{{\mathbb{J}}}
\newcommand{\bbD}{{\mathbb{D}}}

\newcommand{\bbZ}{{\mathbb{Z}}}

  \newcommand{\A}{{\mathcal{A}}}
  \newcommand{\B}{{\mathcal{B}}}

\renewcommand{\H}{{\mathcal{H}}}

  \newcommand{\U}{{\mathcal{U}}}

  \newcommand{\X}{{\mathcal{X}}}

\newcommand{\rC}{{\mathrm{C}}}

\newcommand{\fA}{{\mathfrak{A}}}

\newcommand{\fB}{{\mathfrak{B}}}

\newcommand{\AND}{\text{ and }}

\newcommand{\qand}{\quad\text{and}\quad}

\newcommand{\qfor}{\quad\text{for}\quad}
\newcommand{\qforal}{\quad\text{for all}\quad}

\newcommand{\ad}{\operatorname{ad}}

\newcommand{\id}{{\operatorname{id}}}
\newcommand{\irred}{{\operatorname{irred}}}

\newcommand{\Prim}{\operatorname{Prim}}

\newcommand{\rep}{\operatorname{rep}}

\newcommand{\lip}{\langle}

\begin{document}

\title[Simple C*-algebras]{Semicrossed Products of Simple C*-algebras.}
\thanks{}

\author[K.R.Davidson]{Kenneth R. Davidson}
\address{Pure Math.\ Dept.\\U. Waterloo\\Waterloo, ON\;
N2L--3G1\\CANADA}
\email{krdavids@uwaterloo.ca}

\author[E.G. Katsoulis]{Elias~G.~Katsoulis}
\address{Dept.\ Mathematics\\East Carolina U.\\
Greenville, NC 27858\\USA}
\email{KatsoulisE@mail.ecu.edu}
\begin{abstract}
Let $(\A, \alpha)$ and $(\B, \beta)$ be C*-dynamical systems and
assume that $\A$ is a separable simple C*-algebra and that
$\alpha$ and $\beta$ are $*$-automorphisms.
Then the semicrossed products $\A \times_{\alpha} \bbZ^{+}$ and
$\B \times_{\beta} \bbZ^{+}$ are isometrically isomorphic
if and only if the dynamical systems $(\A, \alpha)$ and $(\B, \beta)$
are outer conjugate.
\end{abstract}

\thanks{First author partially supported by an NSERC grant.}
\thanks{Second author partially supported by
a summer grant from ECU}
\date{}
\maketitle
\section{introduction}

The main objective of this paper is the classification of semicrossed products
of separable simple C*-algebras by an automorphism, up to
isometric isomorphism. It is easily seen (and well-known) that
outer conjugacy between automorphisms of arbitrary C* algebras
is a sufficient condition for the existence of an isometric isomorphism
between the associated semicrossed products.
In this paper we show that for separable simple
C*-algebras, this is also a necessary condition.
This follows from the following general result: if
 $\alpha$ and $\beta$ are automorphisms
of arbitrary C*-algebras $\A$ and  $\B$, then
the presence of an isometric isomorphism from $\A \times_{\alpha} \bbZ^{+}$ onto
$\B \times_{\beta} \bbZ^{+}$ implies the existence of a C*-isomorphism
$\gamma: \A \rightarrow \B$ so that
$\alpha \circ \gamma^{-1}\circ \beta^{-1}\circ \gamma$ is
universally weakly inner with respect to irreducible representations.
The result for simple C*-algebras follows then from a remarkable
result of Kishimoto \cite{Kis} which shows that for
a separable simple C*-algebra, all universally weakly inner automorphisms
are actually inner.

Let $(\A, \alpha)$ be a (discrete) C*-dynamical system, i.e., a C*-algebra
$\A$ together with a $*$-endomorphism $\alpha$ of $\A$.
Motivated by a construction of Arveson \cite{Arv},
Peters \cite{Pet} introduced the concept of the semicrossed
product $\A \times_{\alpha} \bbZ^{+}$.
This is the universal  operator algebra for contractive
covariant representations of this system.

In the commutative case, the
semicrossed product is an algebra
$\rC_0(\X) \times_{\sigma}\bbZ^+$ determined by a dynamical system
$(\X, \sigma)$ given by a proper continuous map $\sigma$ acting on a
locally compact Hausdorff space $X$. Under the assumption that
the topological spaces are compact and the maps are aperiodic, Peters \cite{Pet}
showed that two such semicrossed products are isomorphic
as algebras if and only if the corresponding dynamical systems are conjugate,
thus extending an earlier classification scheme of Arveson \cite{Arv} and
Arveson and Josephson \cite{ArvJ}.
In spite of the subsequent interest in semicrossed products and their
variants \cite{AlaP, Bu, BP, HPW, Lam-1, Lam, MM, OS, Pow, Solel-1},
the problem of classifying semicrossed products of the form
$\rC_0(\X) \times_{\sigma}\bbZ^+$ remained open in the generality introduced by
Peters in \cite{Pet} until our recent paper \cite{DavK}, which established that the
Arveson-Josephson-Peters classification scheme holds with
no restrictions on either $\X$ or $\sigma$.
It was our desire to apply the techniques of \cite{DavK}
and \cite{DK2} to more general settings that
motivated the research of the present paper.

The present paper provides for the first time a classification scheme
for semicrossed products which is valid for a broad class of C*-algebras,
without posing any restrictions on the automorphisms involved.
Our result complements a similar result of Muhly and
Solel~\cite[Theorem~4.1]{MS} regarding semicrossed products with
automorphisms having full Connes spectrum.
In Theorem \ref{MS}, we give an alternative proof of their result
using representation theory.
Both results seem to indicate that outer conjugacy is a complete
invariant for isometric isomorphisms between arbitrary semicrossed products.
They also suggest the problem of establishing the
validity of the conclusion under the weaker requirement of an
\textit{algebraic} isomorphism instead of an isometric isomorphism.

\section{Preliminaries}

Let $\A$ be a C*-algebra and $\alpha$ an endomorphism of $\A$.
The \textit{skew polynomial algebra}
$P(\A, \alpha)$ consists of all polynomials of the form
$\sum_n U_{\alpha}^n A_n$, $A_n\in \A$, where the multiplication of the
``coefficients'' $A\in \A$ with the ``variable'' $U_{\alpha}$ obeys the rule
\[
AU_{\alpha} =   U_{\alpha}\alpha(A)
\]
Equip $P(\A, \alpha)$ with the $l^1$-norm
\[
 \big\| \sum_n U_{\alpha}^n A_n \big\|_1 \equiv \sum_n \| A_n\|
\]
and let $l^1(\A, \alpha)$ be the completion of $P(\A, \alpha)$ with respect to $\|\, .\,\|_1$.

An \textit{(isometric) covariant representation} $(\pi, V)$ of $(\A, \alpha)$ consists of
a C*-representation $\pi$  of $\A$ on a Hilbert space $\H$ and an isometry $V$ on $\H$ so that
$\pi(A)V= V\pi(\alpha(A))$, for all $A \in \A$. Each covariant representation
induces a representation $\pi\times V$ in an obvious way.

\begin{defn}
For $P \in l^1(\A, \alpha)$ let
\[
 \|P\| := \sup  \big\{ \|(\pi\times V) (P)\| :  \pi\times V\
\mbox{ is covariant} \big\}
\]
where $\pi\times V$ runs over all isometric covariant representations of $(\A, \alpha)$.
The \textit{semicrossed product} $\A \times_{\alpha} \bbZ^{+}$
of $\A$ by $\alpha$ is the completion of $l^1(\A, \alpha)$
with respect to this  norm.
\end{defn}

For an alternative description of $\A \times_{\alpha} \bbZ^{+}$, one
may start by obtaining a universal covariant representation
$(\pi\times V)$ of $(\A , \alpha) $ and then define $\A \times_{\alpha} \bbZ^{+}$
to be the non-sefadjoint operator algebra generated $\pi(\A)$ and $V$.
The two constructions produce isomorphic algebras.
For each covariant representation $(\pi, V)$ of $(\A, \alpha)$,
the representation $\pi\times V$ extends uniquely to a contractive
representation of $\A \times_{\alpha} \bbZ^{+}$,
which will also be denoted as $\pi\times V$.

In this paper, we will exclusively
work with \textit{invertible} C*-dynamical systems, i.e.,
the endomorphism will actually be an automorphism.
Therefore we now drop the adjective ``invertible'',
and by C*-dynamical system we will mean an invertible one.
Given an (invertible) dynamical system $(\A, \alpha)$, the
unitary covariant representations for $(\A, \alpha)$ suffice
to capture the norm for $\A \times_{\alpha} \bbZ^{+}$.
Therefore, $\A \times_{\alpha} \bbZ^{+}$ is a natural
nonselfadjoint subalgebra of the crossed product C*-algebra
$\A \times_{\alpha} \bbZ$.

\begin{defn}
Two C*-dynamical systems $(\A, \alpha)$ and $(\B, \beta)$ are said to
be \textit{outer conjugate} if there
exists a C*-isomorphism $\gamma : \A \rightarrow \B$
and a unitary $W \in M(\A)$, the multiplier algebra of $\A$, so
that
\[ \alpha = \ad_W \gamma^{-1} \circ \beta \circ \gamma .\]
\end{defn}

The main issue in this paper is the classification of semicrossed products up to
isometric isomorphism. The following elementary result shows that the outer
conjugacy of automorphisms provides a sufficient condition for the
existence of such an isomorphism.

\begin{prop}  \label{onedir}
If the dynamical systems
$(\A, \alpha)$ and $(\B, \beta)$ are outer conjugate,
then the semicrossed products $\A \times_{\alpha} \bbZ^{+}$
and $\B \times_{\beta} \bbZ^{+}$ are isometrically isomorphic.
\end{prop}
 
\begin{proof}
Without loss of generality assume that $\A = \B$ and $\gamma= \id$.
Let $W \in M(\A)$ so that $\alpha(A) = W\beta(A)W^*$.
Observe that $\beta$ has a unique extension to an automorphism $\bar\beta$
of $M(\A)$ such that
\[ \bar\beta(M) \beta(A) = \beta(MA) \qforal A \in \A \AND M \in M(\A) ,\]
namely, $\bar\beta(M) A = \beta(M\beta^{-1}(A))$.
Therefore $\A \times_\beta \bbZ$ is naturally a subalgebra of
$M(\A) \times_{\bar\beta} \bbZ$ generated by $\A$ and the
universal unitary $U_\beta$ satisfying
$U_\beta MU_\beta^* = \beta(M)$ for $M \in M(\A)$.

Now notice that $\ad_{WU_\beta}$ implements $\bar\alpha$, the extension of
$\alpha$ to $M(\A)$ because it is an automorphism which acts as $\alpha$
on $\A$.  Whence it carries $M(\A)$ to itself, and is the unique extension
of $\alpha$ to $M(\A)$.
Therefore $\ca(\A, WU_\beta)$ determines a representation $\sigma$ of
the C*-algebra crossed product $\A \times_\alpha \bbZ$ given
by $\sigma|_\A = \id$ and $\sigma(U_\alpha) = U_\beta W$.
Since this representation is faithful on $\A$, it follows from the
gauge invariance uniqueness theorem \cite[Theorem 6.4]{Kat} that this is a faithful
representation of the crossed product.

Next observe that $\ca(\A, WU_\beta) = \A \times_\beta \bbZ$.
The point is that
\begin{align*}
 A (U_\beta W)^n &=
 A (U_\beta W U_\beta^*) (U_\beta^2 W U_\beta^{*2})\cdots
 (U_\beta^n W U_\beta^{*n}) U_\beta^n  \\ &=
 A \bar\beta(W) \bar\beta^2(W) \cdots \bar\beta^n(W) U_\beta^n \\ &=
 B U_\beta^n
\end{align*}
where $B = A \bar\beta(W) \bar\beta^2(W) \cdots \bar\beta^n(W)$ belongs to $\A$.
It now follows that $\A (U_\beta W)^n = \A U_\beta^n$.
Hence $\ca(\A, WU_\beta) = \A \times_\beta \bbZ$.
Moreover, the nonself-adjoint subalgebra $\A \times_\beta \bbZ^+$ generated by
$\A$ and $\A U_\beta$ coincides with the algebra generated by
$\A$ and $\A U_\beta W$.  But this latter algebra is canonically
identified with $\A \times_\alpha \bbZ^+$ via the identification of
$\ca(\A, WU_\beta)$ with $\A \times_\alpha \bbZ$.
\end{proof}

Note that $(\A, \alpha)$ and $(\B, \beta)$ are outer conjugate if  and only
if there exists a C*-isomorphism $\gamma : \A \rightarrow \B$ so that
the automorphism
$a\circ \gamma^{-1}\circ\beta^{-1}\circ\gamma$ is inner. A
notion weaker than that of outer congucacy arises from
the concept of a universally weakly inner automorphism. We say an automorphism
$a$ of $\A$ is \textit{universally weakly inner} with
respect to irreducible (resp. faithful)
representations,
if for any irreducible (resp. faithful)
representation $\pi$ of $\A$ on $\H$, there exists a unitary $W \in \pi(\A)^{''}$
so that $\pi(\alpha(A))=W^*\pi(A)W$, $A \in \A$.
The concept of a universally weakly inner automorphism
with respect to faithful representations (or $\pi$-inner automorphism) was introduced
by Kadison and Ringrose
\cite{KadR} and has been studied by various authors \cite{Elliott, Lance}.
Here we will be making use of universally weakly inner automorphisms
with respect to \textit{irreducible} representations.
A direct integral decomposition argument shows that the two concepts
coincide for type I C*-algebras.
Kishimoto \cite{Kis} has shown that for a separable simple C*-algebra
all universally weakly inner automorphisms with respect to
irreducible representations are actually inner.
Therefore the two concepts coincide there as well.
 
We also need to recall several facts for the various concepts of
spectrum from C*-algebra theory. Let $(\A, \alpha)$ be a dynamical system and let
$\Prim A$ be the primitive ideal space of $\A$ equipped with the hull-kernel topology.
For each cardinal $i$, we fix a Hilbert space $\H_i$ with dimension $i$. Then,
$\irred(\A, \H_i)$ is the collection of all irreducible representations on $\H_i$
and $\irred\A\equiv \cup_i\, \irred(\A, \H_i)$. If $\rho \in \irred \A$, then
$[\rho]$ denotes its equivalence class, with respect to unitary equivalence $\simeq$
between representations on the same Hilbert space. Let the \textit{spectrum} of $\A$ be
\[
 \hat{\A}\equiv  \big\{ [\rho] \mid  \rho \in \irred \A \big\}
\]
and consider the canonical map
\[
 \theta :  \hat{\A} \longrightarrow \Prim \A : [\rho] \longrightarrow \ker\rho.
\]
In what follows we always consider $\hat{\A}$ equiped with the
smallest topology that makes
$\theta$ continuous.
For any C*-isomorphism $\gamma: \A \rightarrow \B$, we define a map
$\hat{\gamma}: \hat{\A} \rightarrow \hat{\B}$ between the
corresponding spectra by the formula $\hat{\gamma}([\rho])= [\rho\circ \gamma]$.

The following result is straightforward.

\begin{lem} \label{theta}
Let $\A$ be a C*-algebra and let $\X\subseteq \hat{\A}$ have empty interior. Then,
\[
 \bigcap_{x \in \hat{\A}\setminus \X}\theta(x)=\{0\}.
\]
\end{lem}
 
Another notion of spectrum is the Connes spectrum. We do not give the precise definition
but instead state the fact \cite[Theorem 10.4]{OPed} that for a separable C*-algebra,
$(\A, \alpha)$ has full Connes spectrum if and only if there is
a dense $\alpha$-invariant subset $\Delta_{\alpha} \subseteq \hat{\A}$ on
which $\hat{\alpha}$ is freely acting.  This is equivalent to the fact that
the periodic points of $\hat{\alpha}$ with period $n$ has no interior for any $n \ge1$.

\section{The main result}

We begin this section with a general result about isometric isomorphisms
between arbitrary operator algebras which is well-known.
 
\begin{prop}   \label{diag}
Let $\phi: \fA\rightarrow \fB$ be an isometric isomorphism between operator algebras.
Then $\phi(\fA \cap \fA^*)= \fB \cap \fB^*$ and $\phi\vert_{\fA \cap \fA^*}$
is a C*-isomorphism.
\end{prop}
 
\begin{proof}
The unitary operators in $\fA \cap \fA^*$ are characterized as the norm 1 elements
$A \in \A$ so that $A^{-1} \in \fA$ and $\| A^{-1}\|=1$.
 From this it follows that $\phi$
maps the unitary group of $\fA \cap \fA^*$ onto the unitary group
of $\fB \cap \fB^*$, and this proves the first assertion.
The second follows from the fact that $\phi$ preserves inverses of unitaries in
$\fA \cap \fA^*$ and hence adjoints.
\end{proof}
 
In order to prove the main result, we need to use representation theory.
In light of Proposition \ref{diag}, it suffices to consider representations
that preserve the diagonal. Hence, if $\fA$ is an operator algebra,
then $\rep(\fA, \H)$ will denote the collection of all contractive representations of
$\fA$ on $\H$ whose restriction on the diagonal $\fA \cap\fA^*$ is a $*$-homomorphism.

\begin{defn}  \label{nestrepn}
If $\fA$ is an operator algebra and $\H$ a Hilbert space,
then $\rep_2(\fA, \H)$ will denote the collection of
all (contractive) representations $\rho \in \rep(\fA, \H \oplus \H)$ of the form
\[
\rho(A)=
\begin{bmatrix} \rho_1(A) & \rho_3(A) \\  0 & \rho_2(A) \end{bmatrix} \qforal A \in \fA,
\]
so that $\rho_i \vert_{\fA \cap \fA^*} \in \irred(\fA \cap \fA^*, \H)$
for $i=1,2$ and $\rho_3(\fA)\neq \{0\}$.
\end{defn}

\begin{lem} \label{basic}
Let  $(\A, \alpha)$ be a C*-dynamical system and let
\[
\rho \in \rep_2(\A\times_{\alpha}\bbZ^+, \H).
\]
Then, $\rho_1\vert_{\A} \simeq \rho_2 \vert_{\A}\circ\alpha$.
\end{lem}

\begin{proof} First note that $\rho$ is a $*$-representation on the
diagonal $\A$. Hence for each $A \in \A$, $\rho(A)$ necessarily has
its $(1, 2)$-entry equal to zero, i.e., $\rho(\A)$ is in diagonal form.

Let$\{E_j\}_j$ be an approximate unit for $\A$;
and let $X, Y, Z\in \H$ so that $\{\rho(U_{\alpha}E_j)\}_j$
converges weakly to
$\left[\begin{smallmatrix} X &Y \\0&Z \end{smallmatrix}\right]$.
\textit{We claim that $Y\neq0$.}
Indeed, otherwise the equality
\[
 \rho(U_{\alpha}A) = \lim_j \rho(U_{\alpha}E_jA)
 = \lim_j \rho(U_{\alpha}E_j)\rho(A) \qforal A \in \A
\]
would imply that $\rho(U_{\alpha}\A)$ is diagonal; and therefore
$\rho(\A\times_{\alpha}\bbZ^+)$ is in diagonal form,
contradicting the requirement $\rho_{3}(\A\times_{\alpha}\bbZ^+)\neq \{0\}$.

Now notice that for any $A \in \A$, we have
\begin{align*}
\rho(A)\lim_j \rho(U_{\alpha}E_j)& = \lim_j \rho(U_{\alpha}\alpha(A)E_j) \\
                        &=\lim_j(U_{\alpha}E_j)\rho(\alpha(A)).
\end{align*}
Hence in matricial form,
\begin{equation*}
 \begin{bmatrix} \rho_1(A) & 0\\ 0 & \rho_2(A) \end{bmatrix}
 \begin{bmatrix} X& Y \\ 0 & Z \end{bmatrix}
=
 \begin{bmatrix} X & Y \\ 0 & Z \end{bmatrix}
 \begin{bmatrix} \rho_1(\alpha(A))&  0\\  0 & \rho_2(\alpha(A)) \end{bmatrix} .
\end{equation*}
By multiplying and comparing $(1,2)$-entries, we obtain
\[
\rho_1(A)Y=Y\rho_2(\alpha(A)).
\]
Since $Y\neq0$ and $\rho_1\vert_{\A}$ and $\rho_2\vert_{\A}$ are irreducible,
this implies that  $\rho_1\vert_{\A} \simeq\rho_2\vert_{\A}\circ\alpha$, as desired.
\end{proof}

\begin{lem}  \label{itexists}
Let  $(\A, \alpha)$ be a C*-dynamical system, and let
$\sigma$ belong to $\irred(\A, \H)$. Then
there exists a representation $\rho \in \rep_2(\A\times_{\alpha}\bbZ^+, \H)
$ so that $\rho_2\vert_{\A}=\sigma$.
\end{lem}

\begin{proof}
Let
\[
\rho(A)=
 \begin{bmatrix} \sigma\circ\alpha(A) & 0 \\ 0 & \sigma(A) \end{bmatrix}
 \qand
 \rho(U_{\alpha}A)=
 \begin{bmatrix} 0 & \sigma(A) \\ 0 & 0  \end{bmatrix} .
\]
It is easily verified that
$\big( \rho\vert_{\A}, \left[\begin{smallmatrix} 0 &I \\0&0 \end{smallmatrix}\right] \big)$ is a
covariant representation and so
$\rho$ extends to a representation of $\A\times_{\alpha}\bbZ^+$.
\end{proof}

Note that the representation $\rho$ in Lemma \ref{itexists} satisfies
$\rho_1|_{\A}= \mbox{$\sigma\circ \alpha$}$.
This is not just an artifact of our construction.
By Lemma \ref{basic}, any representation
in $\rho \in \rep_2(\A\times_{\alpha}\bbZ^+, \H)$ will satisfy that property, provided
that $\rho_2\vert_{\A}=\sigma$.

The following general result relates the classification problem
for semicrossed products to the study of universally weakly inner automorphisms for
C*-algebras.

\begin{thm}   \label{main}
Let $(\A, \alpha)$ and $(\B, \beta)$ be C*-dynamical systems, and
assume that the semicrossed products $\A \times_{\alpha} \bbZ^{+}$ and
$\B \times_{\beta} \bbZ^{+}$ are isometrically isomorphic. Then
there exists a C*-isomorphism $\gamma: \A\rightarrow\B$
so that $\alpha\circ\gamma^{-1}\circ \beta^{-1}\circ \gamma$ is
universally weakly inner with respect to irreducible representations.
\end{thm}

\begin{proof}Assume that there exists an isometric isomorphism
\[
\gamma: \A \times_{\alpha} \bbZ^{+}\longrightarrow \B \times_{\beta} \bbZ^{+}.
\]
By Proposition \ref{diag}, $\gamma\vert_{\A}$ is a $*$-isomorphism onto $\B$,
which we will also denote by $\gamma$; this is the promised isomorphism.
Indeed $\gamma$ establishes a correspondence
\[
\irred(\A) \ni \sigma \longrightarrow \sigma \circ \gamma^{-1}\in \irred(\B)
\]
that preserves equivalence classes. To show that $\alpha\circ\gamma^{-1}\circ \beta^{-1}\circ \gamma$ is
universally weakly inner with respect to irreducible representations, it is enough to show
that
\[
\sigma \circ\alpha\circ\gamma^{-1} \simeq \sigma \circ \gamma^{-1} \circ \beta
\]
for any $\sigma \in \irred(\A)$.

By Lemma \ref{itexists},
there exists $\rho \in \rep_2(\A\times_{\alpha}\bbZ^+, \H)$
so that $\rho_2\vert_{\A}=\sigma$. But then
\[
\rho\circ \gamma^{-1} \in  \rep_2(\B\times_{\beta}\bbZ^+, \H).
\]
Hence, Lemma \ref{nestrepn} implies that
\[
(\rho\circ\gamma^{-1})_1\vert_{\B} \simeq (\rho\circ\gamma^{-1})_2 \vert_{\B} \circ \beta
\]
or, once again, $\sigma\circ \alpha\circ\gamma^{-1}
\simeq \sigma\circ\gamma^{-1}\circ \beta$ .
\end{proof}

We have arrived to the main result of the paper.

\begin{thm} \label{Main}
Let $(\A, \alpha)$ and $(\B, \beta)$ be C*-dynamical systems and
assume that $\A$ is a separable simple C*-algebra.
The semicrossed products $\A \times_{\alpha} \bbZ^{+}$ and
$\B \times_{\beta} \bbZ^{+}$ are isometrically isomorphic
if and only if the dynamical systems $(\A, \alpha)$ and $(\B, \beta)$
are outer conjugate.
\end{thm}

\begin{proof}One direction follows from Proposition \ref{onedir}.
Conversely assume that $\A \times_{\alpha} \bbZ^{+}$ and
$\B \times_{\beta} \bbZ^{+}$ are isometrically isomorphic.
Theorem \ref{main} shows that there exists a C*-isomorphism $\gamma: \A\rightarrow\B$
so that $\alpha\circ\gamma^{-1}\circ \beta^{-1}\circ \gamma$ is
universally weakly inner with respect to irreducible representations.
The conclusion follows now from Kishimoto's result \cite[Corollary 2.3]{Kis}.
\end{proof}

\section{Representations and dynamics on the spectrum}

As we mentioned in the introduction, an earlier
result of Muhly and Solel \cite[Theorem 4.1]{MS} implies the validity of
Theorem \ref{Main} for arbitrary $\ca$-algebras,
provided that the automorphisms $\alpha$ and $\beta$ have full Connes spectrum.
In what follows we present an
alternative proof of that result of Muhly and Solel,
based on the ideas developed in this paper.

Let $\hat{\Phi}_{\alpha}$ denote the
canonical expectation from $\A \times_{\alpha} \bbZ$ onto $\A$ and let
$\Phi_{\alpha}$ denote its restriction on the semicrossed product
$\A \times_{\alpha} \bbZ^+$.
The following is the key step in their proof.

\begin{lem}
Let $(\A, \alpha)$ and $(\B, \beta)$ be separable
C*-dynamical systems and let $\gamma :\A \times_{\alpha} \bbZ^{+} \rightarrow
\B \times_{\beta} \bbZ^{+}$ be an isometric
isomorphism. If $\alpha$ has full Connes spectrum,
then $\gamma(\ker \Phi_{\alpha})= \ker \Phi_{\beta}$.
\end{lem}

\begin{proof}Let $\Delta_{\alpha} = \{\sigma_j\mid j \in \bbJ\}$ be the dense
$\alpha$-invariant set of aperiodic points described in the introduction.
By Lemma \ref{theta}, $\oplus_{j}\sigma_j$
is a faithful representation of $\A$ and
so $\left(\oplus_{j}\sigma_j\right) \times_{\alpha} V$
defines a faithful representation of
$\A \times_{\alpha} \bbZ$.
Now notice that the
compression on the main diagonal of $\left(\oplus_j \sigma_j\right) \times_{\alpha} V$
defines an expectation from
$\A \times_{\alpha} \bbZ$ onto $\A$ which coincides with $\hat{\Phi}_{\alpha}$.
Similarly, $\left(\oplus_{j}\sigma_j \circ \gamma^{-1}\right) \times_{\beta} V$
defines a faithful representation of $\B \times_{\beta} \bbZ$ and
$\hat{\Phi}_{\beta}$ is the compression on the main diagonal.

Now $\gamma(\A) = \B$ and
$\ker \Phi_{\alpha}=U_{\alpha}\left(\A \times_{\alpha} \bbZ^{+}\right)$. Hence it is enough to show that
$\gamma(U_{\alpha}) \in \ker \Phi_{\beta}$, i.e.,
the diagonal entries of $\gamma(U_{\alpha})$ in the representation
$\left(\oplus_{j}\sigma_j \circ \gamma^{-1}\right) \times_{\beta} V$
of $\B \times_{\beta} \bbZ^{+}$ are equal to 0. To verify this
examine these entries in light of the covariance equation as in Lemma \ref{basic}.
\end{proof}

The rest of the proof follows the same arguments as in \cite{MS}.

\begin{thm}[\cite{MS}]   \label{MS}
Let $(\A, \alpha)$ and $(\B, \beta)$ be separable C*-dynamical systems and
assume that $\alpha$ has full Connes spectrum. The semicrossed
products $\A \times_{\alpha} \bbZ^{+}$ and
$\B \times_{\beta} \bbZ^{+}$ are isometrically isomorphic if and only if
the dynamical systems $(\A, \alpha)$ and $(\B, \beta)$ are outer conjugate.
\end{thm}

\begin{proof}
Assume that there is an isometric isomorphism
$\gamma$ of $\A \times_{\alpha} \bbZ^{+}$ onto $\B \times_{\beta} \bbZ^{+}$.
By the previous Lemma we have
\begin{align*}
 \gamma(U_{\alpha})&= BU_{\beta}+Y,
  \quad\text{where }B \in \B \AND Y \in U_{\beta}^2 (\B \times_{\beta} \bbZ^{+}) ,\\
 \gamma^{-1}(U_{\beta})&= AU_{\alpha} + Z,
  \quad\text{where }A \in \A \AND  Z \in U_{\alpha}^2 (\A \times_{\alpha} \bbZ^{+}).
\end{align*}
Since both $\gamma$ and $\gamma^{-1}$ are isometries, $\|A\|, \|B\|\leq1$.
Also,
\[
 U_{\alpha}= \gamma^{-1}(\gamma(U_{\alpha}))
 =\gamma^{-1}(B)AU_{\beta} +\gamma^{-1}(B)Z+\gamma^{-1}(Y),
\]
which by the uniqueness of the Fourier expansion implies that \break
$\gamma^{-1}(B)A=I$.
Since $A$ and $\gamma^{-1}(B)$ are contractions, they must both be unitary.
Hence $B$ and $\gamma(A)$ are also unitary.
Therefore
\begin{align*}
 \hat{\Phi}_{\beta}(\gamma(U_{\alpha})^*\gamma(U_{\alpha})) &=
 \hat{\Phi}_{\beta}(I + U_\beta^* B^*Y + Y^* B U_\beta + Y^*Y) \\&=
 I+ \hat{\Phi}_{\beta}(Y^*Y)
\end{align*}
since
\[
 U_\beta^* B^*Y \in U_\beta^* \B U_{\beta}^2 (\B \times_{\beta} \bbZ^{+})
 \subseteq U_{\beta} (\B \times_{\beta} \bbZ^{+})  \subseteq \ker\Phi_\beta;
\]
and likewise $Y^* B U_\beta \in (\ker\Phi_\beta)^*\subseteq \ker\hat{\Phi}_\beta$.
So $\hat{\Phi}_{\beta}(Y^*Y)=0$; whence $Y=0$. Hence, $\gamma(U_{\alpha})= BU_{\beta}$
which implies that $(\A, \alpha)$ and $(\B, \beta)$ are outer conjugate.
\end{proof}

We finish the paper with a result of independent interest that associates the fixed points
of $\hat{\alpha}$ to a certain analytic structure in the space of representations
$\rep(\A \times_{\alpha} \bbZ^+, \H)$.

\begin{defn}Let $\fA$ be an operator algebra.
A map
\[
 \Pi :\bbD \equiv \{z \in \bbC\mid |z|<1\}
 \longrightarrow (\rep(\fA, \H), \text{point--\sot})
\]
is called \textit{analytic} if for each $A \in \fA$ and $x,y \in \H$ the map \break
$z \rightarrow \left< \Pi(z)(A)x\mid y\right>$, $z \in \bbD$,
is analytic in the usual sense.
\end{defn}

Assume now that $(\A , \alpha)$ is a C*-dynamical system,
and let
\[
 \Pi :\bbD
 \longrightarrow \rep(\A\times_{\alpha} \bbZ^+, \H)
\]
be an analytic map. Since $\Pi(z)$, $z \in \bbD$, is a $*$-homomorphism on the diagonal,
we may write
\[
 \Pi(z)= \pi_z \times K_z, \quad z \in \bbD,
\]
where $\pi_z=\Pi(z)\vert _{\A}$ and $K_z=\Pi(z)(U_{\alpha})$.
We claim that the map \break
$ z \longrightarrow \pi_z$, $z \in \bbD$,
is constant. Indeed, consider any selfadjoint
operator $A \in \A$ and notice that the map
\[
 \bbD \ni z \longrightarrow \lip \pi_z(A)x, x \rip, \quad x \in \H,
\]
is real analytic and therefore constant, which proves the claim.

\begin{defn}
Let $\fA$ be an operator algebra. An analytic map $\Pi: \bbD \rightarrow
\rep(\fA, \H)$ is said to be \textit{irreducible on the diagonal} if for
any $z \in \bbD$, the representation $\Pi(z)\vert_{\fA \cap \fA^*}$ is irreducible.
\end{defn}

\begin{prop}   \label{analytic}
Let $(\A, \alpha)$ be a C*-dynamical system, and $\sigma$
belong to $\irred (\A, \H)$.
Then $\sigma \simeq \sigma \circ \alpha$ if and only
if there exists a non-constant analytic and irreducible on the diagonal
map
\[
 \Pi: \bbD \longrightarrow \rep(\A \times_{\alpha} \bbZ^+, \H)
\]
so that
\[
 \Pi(z)\vert_{\A } = \sigma
\]
for some $z \in \bbD$.
\end{prop}

\begin{proof}
First assume that $\sigma \simeq \sigma \circ \alpha$ and let
$U$ be a unitary such that  $\sigma(A) U = U (\sigma \circ \alpha)(A)$
for $A \in \A$.  Then
\[
 \Pi(z) = \sigma \times zU  \qfor z \in \bbD
\]
has the desired properties.

Conversely, assume that such a map $\Pi$ exists. By the discussion above
$\Pi(z)\vert { \A }=\sigma$ for all $z \in \bbD$.  Since $\Pi$ is not constant, there
exists some $z \in \bbD$ so that $K := \Pi(z)(U_{\alpha}) \ne 0$. Therefore,
\begin{align*}
 \sigma(A)K&=\Pi(z)(A)\Pi(z)(U_{\alpha})   \\
 &=\Pi(z)\left(U_{\alpha}\alpha(A)\right)    \\
 &=K\sigma(\alpha(A)) \qquad \qforal A \in \A.
\end{align*}
\noindent Since both $\sigma$ and $\sigma \circ \alpha$ are irreducible and $K \neq 0$,
we obtain that
$\sigma \simeq \sigma\circ \alpha$.
\end{proof}

\begin{rem}
Note that if $\Pi$ is as in the above Proposition, then the operator $K$ in the proof
is necessarilly a scalar multiple of the (unique) unitary operator $U$ implementing
the equivalence $\sigma \simeq \sigma \circ \alpha$.
Therefore the range of $\Pi$ is contained in
\[
 D_{\sigma} \equiv \{ \sigma \times  zU \mid z\in \overline{\bbD}\}.
\]
In this fashion, we associate with each
fixed point $\sigma \in \hat{\A}$, a unique \textit{maximal analytic set} $D_{\sigma}$.
Any representation of $\A \times_{\alpha} \bbZ^+$ not belonging to the union
of the maximal analytic sets is associated with a non-fixed point of $\hat{\alpha}$.
\end{rem}


\end{document}